\documentclass[english]{cccconf}
\usepackage[comma,numbers,square,sort&compress]{natbib}
\usepackage{graphicx,subfigure,amsmath,amssymb,mathrsfs,amsfonts,amstext,amsthm}
\usepackage{latexsym, amssymb}
\usepackage{graphicx}
\usepackage{pgf,tikz}
\usepackage{pstricks}

\newtheorem{rem}{Remark}
\newtheorem{thm}{Theorem}
\newtheorem{lem}{Lemma}
\newtheorem{cor}{Corollary}

\newtheorem{defn}{Definition}

\begin{document}

\title{Free Information Flow Benefits Truth Seeking}%
\author{Wei Su\aref{su},  Yongguang Yu\aref{su}}%

\affiliation[su]{School of Science, Beijing Jiaotong University, Beijing 100044
        \email{suwei@amss.ac.cn, ygyu@bjtu.edu.cn}}
\date{}%
\maketitle
\begin{abstract}
How can we approach the truth in a society? It may depend on various factors. In this paper, using a well-established truth seeking model, we show that the persistent free information flow will bring us to the truth. Here the free information flow is modeled as the environmental random noise that could alter one's cognition. Without the random noise, the model predicts that the truth can only be captured by the truth seekers who own active perceptive ability of the truth and their believers, while the other individuals may stick to falsehood. But under the influence of the random noise, we strictly prove that even there is only one truth seeker in the group, all individuals will finally approach the truth.
\end{abstract}
\footnotetext{This work is supported by the National Natural Science Foundation of China under grants No. 11371049 and the Fundamental Research Funds for the Central Universities (2016JBM070).}
\section{Introduction}\label{Secintro}

In recent years, opinion dynamics has become an emerging research hot and
attracted more and more attention from various areas, including the social science, mathematics, information theory, and so on \cite{Castellano2009,Proskurnikov2017}. The investigation methods of opinion dynamics range from qualitative ways to quantitative ways, and the latter one has become more and more important in the rapidly developing information age. One of the main tools for quantitatively investigating the evolution of opinions is building the mathematical model to analyzing the behaviors of the opinion evolution, and several agent-based mathematical models of opinion dynamics were proposed in the past decades \cite{DeGroot1974,Friedkin1999,Deffuant2000,Krause2000,Hegselmann2002,Friedkin2016}. It is interesting that these models could effectively characterize the complexity of opinion behaviors in plenty of situations and display various phenomena, such as the basic agreement or disagreement of the opinions \cite{Hegselmann2002}, and the more complex autocratic or democratic social structures \cite{Jia2015}.

Among the diverse issues concerned in the social dynamics, an interesting topic is to investigate how we can approach truth in a social opinion group \cite{Hegselmann2006,Malarz2006}. We want to know what kind of factors could be beneficial for us to seek truth. In \cite{Hegselmann2006}, a truth seeking model was built based on the well-known Hegselmann-Krause (HK) confidence-based opinion dynamics (also see Section \ref{Secmodel} in the current paper). In this model, various factors are shown to be connected with the truth seeking of the group, including the proportion of the truth seekers, the attraction strength of the truth to truth seekers, the confidence threshold of the individuals, and the position of the truth value and the initial opinion values of all individuals. Nevertheless, how these factors determine the truth seeking is intricate and even counter-intuitive somehow. For example, intuitively, the more truth seekers exist and the stronger the strength of the truth attraction is, the more likely the group will approach the truth if given other conditions unchanged. However, the situation is far more complex, since it can be shown that the more truth seekers and the larger attraction strength could make a reachable truth unreachable.

Apart from the factors mentioned above, are there any other remarkable factors that could affect the truth seeking? For the social opinion dynamics, an inevitable influence factor is the free information flow in the society, especially in this Internet era. People's opinions are affected not only by communicating with the other individuals, but also by receiving the intensive information from the free media, such as the newspapers, the TV shows, the broadcast, and especially the information flow through the social media nowadays. Usually this free information flow can be modeled as persistent random noise added to the opinion dynamics model. By simulating and analyzing the noisy opinion models, it is interesting to find that the random noise could play a positive role in enhancing the consensus of opinions \cite{Mas2010,Grauwin2012,Carro2013, Pineda2013}. A strict theoretical analysis of this fact was established very recently based on the HK model in \cite{Su2016}, where it is proved that when all the individuals were affected by the random noises, the noisy HK model will almost surely reach quasi-consensus (a consensus concept defined in the noise case) if the noise strength is below a critical value, while the noisy opinions almost surely diverge if the noise strength exceeds the critical value. This fact reveals that the free information could affect the social opinion evolution essentially.

With the above exciting findings about the free information flow, we also eager to know whether the free information flow could benefit the truth seeking in a society. In this paper, using the truth seeking model built in \cite{Hegselmann2006}, we indeed find that the free information flow will effectively help the individuals acquire the truth. To be specific, we strictly prove that, under the influence of the random noise whose strength is below a proper value, the opinion values of all individuals will almost surely approach the truth value and stay nearby, as long as there is only one active truth seeker or more in the group. The effective noise strength is shown to only depend on the intrinsic parameters of the group, i.e., the number of the truth seekers, the size of the group, the attraction strength of the truth, and the confidence threshold.

The rest of this paper is organized as follows: Section \ref{Secmodel} presents the truth seeking model with free information flow and Section \ref{Secresult} gives the main analytical results of the paper. Some numerical simulations are displayed in Section \ref{Secsimul} to verify our main results, and Section \ref{Secconc} concludes the paper finally.

\section{Preliminary and Formulation}\label{Secmodel}
\subsection{Truth Seeking Model}
The noise-free truth seeking model was established in \cite{Hegselmann2006}: for $i\in\mathcal{V}$,
\begin{equation}\label{basicTSmodel}
  x_i(t+1)=\alpha_iA+(1-\alpha_i)\sum\limits_{j\in\mathcal{N}_i(x(t))}\frac{x_j(t)}{|\mathcal{N}_i(x(t))|},
\end{equation}
Here, $\mathcal{V}=\{1,\ldots,n\}$ is the group of agents, $x_i(t)\in[0,1]$ is opinion value of agent $i$ at time $t$, $\epsilon$ is the common confidence threshold, $A\in[0,1]$ is the truth value, $\alpha_i\in[0,1]$ is the attraction strength of truth for agent $i$, and
\begin{equation}\label{neigh}
 \mathcal{N}_i( x(t))=\{j\in\mathcal{V}\; \big|\; |x_j(t)-x_i(t)|\leq \epsilon\}
\end{equation}
is the neighbor set of $i$ at time $t$.

Denote $\mathcal{S}=\{i\in\mathcal{V}:0<\alpha_i\leq 1\}$ as the set of truth seekers. When $|\mathcal{S}|=0$, i.e., there is no truth seeker in the group, the model (\ref{basicTSmodel}) degenerates to the classical Hegselmann-Krause opinion dynamics \cite{Hegselmann2002}, where opinion fragmentation may happen, let alone failing in seeking truth. When $|\mathcal{S}|=n$, i.e., all the individuals are truth seekers, the truth is surely to be achieved \cite{Hegselmann2006}. If $1\leq|\mathcal{S}|<n$, i.e., only part of agents are truth seekers, the truth can approached by the truth seekers and some of their neighbors, while the other individuals cannot acquire the truth \cite{Hegselmann2006} (see Figure \ref{noisefreefig} for illustration).

\subsection{Free Information Flow}
In real situations, people's attitudes or beliefs can also be influenced by the touched information, from the news, the broadcast, the TV shows, and even some rumors. This information flow affects one's opinion in a stochastic way. Hence, the truth seeking model (\ref{basicTSmodel}) should be added some random noises: for $t\geq 0$,
\begin{equation}\label{noiseTSmodel0}
\begin{split}
  x_i(t+1)=&(1-\alpha_i I_{\{i\in\mathcal{S}\}})\sum\limits_{j\in\mathcal{N}_i(x(t))}\frac{x_j(t)}{|\mathcal{N}_i(x(t))|}\\
  &+\alpha_i I_{\{i\in\mathcal{S}\}}A+\xi_i(t+1),\,\,\,i\in\mathcal{V},
\end{split}
\end{equation}
where $I_{\{i\in\mathcal{S}\}}$ is the indicator function, which takes value 1 or 0, according to $i\in\mathcal{S}$ or not. $\{\xi_i(t)\}_{i\in\mathcal{V},t> 0}$ are the random noises that model the free information flow and we simply assume they are the independent random variables uniformly distributed on $[-\delta,\delta]$ with $\delta>0$.

On account of the extreme opinions in reality, in our model settings in the following, we suppose the noisy opinion values are still limited in a closed interval, say $[0,1]$ without loss of generality. Also, in our analysis, we only consider the homogenous attraction strength, i.e., $\alpha_i=\alpha\in(0,1], i\in\mathcal{V}$, then
\begin{equation}\label{TSnoise1}
  x_i(t+1)=\left\{
           \begin{array}{ll}
             1,  & \hbox{~$x_i^*(t)>1~$} \\
             x_i^*(t),  & \hbox{~$x_i^*(t)\in[0, 1]~$} \\
             0,  & \hbox{~$x_i^*(t)<0~$}
           \end{array},~\forall i\in\mathcal{V}, t\geq 0,
         \right.
\end{equation}
where
\begin{equation}\label{xit}
\begin{split}
  x_i^*(t)=&(1-\alpha I_{\{i\in\mathcal{S}\}})\sum\limits_{j\in\mathcal{N}_i(x(t))}\frac{x_j(t)}{|\mathcal{N}_i(x(t))|}\\
  &+\alpha I_{\{i\in\mathcal{S}\}}A+\xi_i(t+1).
\end{split}
\end{equation}
and
\begin{equation}\label{precond}
\begin{split}
  &1\leq |\mathcal{S}|\leq n,\\
  & \{\xi_i(t),i\in\mathcal{V}, t>0\}\,\text{are independent and uniformly}\\
  & \text{distributed on}\,[-\delta,\delta].
\end{split}
\end{equation}

\subsection{Approach Truth}
For the truth seeking model (\ref{basicTSmodel}), the system can be said to acquire the truth $A$ if $\lim\limits_{t\rightarrow\infty}x_i(t)=A, i\in\mathcal{V}$. However, with the presence of random noise, which could fluctuate the opinions with the amplitude of the noise strength, we should give a definition for the system (\ref{TSnoise1})-(\ref{precond}) to achieve the truth approximately:
\begin{defn}\label{tsphidef}
Denote $d_\mathcal{V}^{A}(t)=\max\limits_{i\in\mathcal{V}}|x_i(t)-A|$, then we say the system (\ref{TSnoise1})-(\ref{precond}) approach the truth with $\phi$-precision, if $\limsup\limits_{t\rightarrow\infty}d_\mathcal{V}^{A}(t)\leq \phi$ a.s. holds.
\end{defn}

\section{Main Results}\label{Secresult}
As Figure \ref{noisefreefig} shows, it is also easy to check that, as long as not all the individuals are truth seekers in the group, there always exist initial conditions that the system cannot achieve the truth. In this part, we will show that,  with the free information flow, the system could approach the truth given any initial conditions.

Let $m=|\mathcal{S}|$ and denote
\begin{equation}\label{parameters}
  \begin{split}
&\delta_1=\frac{n(1-\alpha)\delta}{m\alpha}+\delta,\, \delta_2=\frac{n\delta}{m\alpha}+\delta, \,\bar{\delta}=\max\{\delta_1,\delta_2\},\\ & \underline{\delta}=\min\{\frac{m\alpha}{2n+(2m-n)\alpha}\epsilon,\frac{m}{n+2m}\epsilon\}
  \end{split}
\end{equation}
then for system (\ref{TSnoise1})-(\ref{precond}), we have the following main result:
\begin{thm}\label{TSthm}
Let $0<\alpha<1$ in the model (\ref{TSnoise1})-(\ref{precond}), then for any initial state $x(0)\in[0,1]^n, \epsilon\in(0,1]$ and $\delta\in(0,\underline{\delta}]$, the system will a.s. approach the truth with $\bar{\delta}$-precision.
\end{thm}
\begin{rem}\label{tsthmrem}
The truth seeking of the model (\ref{basicTSmodel}) is related to the ``intrinsic parameters'', i.e., $\epsilon, \alpha, n, m$, of the system, however, this relationship is quite intricate as pointed out in Section \ref{Secintro}. With the free information flow, Theorem \ref{TSthm} indicates that the system will approach the truth and the effective noise strength is also determined by the intrinsic parameters. However, the relationship between the parameters and the truth seeking result, $\bar{\delta}$ and $\underline{\delta}$, is much intuitive. Roughly, the larger $m$ and $\alpha$ the system possesses, the larger the effective noise strength $\underline{\delta}$ is allowed, and the better the precision $\bar{\delta}$ is obtained for given $\delta$.
\end{rem}
The following Corollary provides a special case of Theorem \ref{TSthm} for a better demonstration:
\begin{cor}\label{TSScoro}
Suppose $m=\lceil n/2\rceil,\alpha=0.5$, then a.s. $\limsup\limits_{t\rightarrow\infty}d_\mathcal{V}^{A}(t)\leq 5\delta$ for all $\delta\in(0,\epsilon/8]$.
\end{cor}
Now we turn to the proof of Theorem \ref{TSthm}, and some preliminary lemmas are needed.
\begin{lem}\cite{Blondel2009}\label{monosmlem}
Suppose $\{z_i, \, i=1, 2, \ldots\}$ is a nondecreasing (nonincreasing) real sequence.  Then for any integer $s\geq 0$, the sequence
$\{g_s(k)=\frac{1}{k}\sum_{i=s+1}^{s+k}z_i, k\geq 1\}$ is monotonically nondecreasing (nonincreasing) with respect to $k$.
\end{lem}
\begin{lem}\label{tslem}
For $0<\alpha\leq 1$, if there is finite time $T$ such that $d_\mathcal{S}^{A}(T)\leq \delta_1, d_{\bar{\mathcal{S}}}^A(T)\leq\delta_2 (\bar{\mathcal{S}}=\mathcal{V}-\mathcal{S})$, then a.s. $\limsup\limits_{t\rightarrow\infty}d_\mathcal{S}^{A}(T)\leq \delta_1, \limsup\limits_{t\rightarrow \infty}d_\mathcal{\bar{S}}^A(T)\leq\delta_2$ on $\{T<\infty\}$ for $0<\delta\leq \underline{\delta}$.
\end{lem}
\begin{proof}
At time $T$, we have a.s.
\begin{equation}\label{neighT}
\begin{split}
  &\max\limits_{i,j\in\mathcal{V}}|x_i(T)-x_j(T)|\\
  \leq &\max\limits_{i\in\mathcal{V}}|x_i(T)-A|+\max\limits_{j\in\mathcal{V}}|x_j(T)-A|\\
  \leq &\delta_1+\delta_2\leq \epsilon.
\end{split}
\end{equation}
This means all agents are neighbors to each other at step $T$, implying $|\mathcal{N}_i(x(T))|=n, i\in\mathcal{V}$. Hence by (\ref{TSnoise1}) and (\ref{xit}),
\begin{equation}\label{opinT1}
  \begin{split}
  x_i(T+1)=&\alpha A+(1-\alpha)\frac{\sum_{k\in\mathcal{V}}x_k(T)}{n}+\xi_i(T+1),\\
  &\qquad\qquad\qquad i\in\mathcal{S},\\
  x_j(T+1)=&\frac{\sum_{k\in\mathcal{V}}x_k(T)}{n}+\xi_j(T+1),\quad j\in \bar{\mathcal{S}}.
  \end{split}
\end{equation}
Then a.s.
\begin{equation}\label{dvaT11}
  \begin{split}
 & |x_i(T+1)-A|\\
 =&\bigg|\alpha A-A+\frac{\sum_{k\in\mathcal{V}}x_k(T)}{n}+\xi_i(T+1)\bigg|\\
  \leq & \frac{1-\alpha}{n}\sum_{k\in\mathcal{V}}|A-x_k(T)|+\delta\\
  \leq & (1-\alpha)\frac{m\delta_1+(n-m)\delta_2}{n}\delta+\delta\leq \delta_1,\quad i\in\mathcal{S}.
  \end{split}
\end{equation}
and a.s.
\begin{equation}\label{dvaT12}
  \begin{split}
  &|x_j(T+1)-A|\\
  =&\bigg|\frac{\sum_{k\in\mathcal{V}}x_k(T)}{n}-A+\xi_i(T+1)\bigg|\\
  \leq & \frac{1}{n}\sum_{k\in\mathcal{V}}|A-x_k(T)|+\delta\\
  \leq & \frac{m\delta_1+(n-m)\delta_2}{n}\delta+\delta\leq \delta_2,\quad j\in\bar{\mathcal{S}}.
  \end{split}
\end{equation}
This means
\begin{equation*}
  d_\mathcal{S}^{A}(T+1)\leq \delta_1,\quad d_\mathcal{\bar{S}}^A(T+1)\leq\delta_2.
\end{equation*}
Repeating the above procedure, we will obtain the conclusion.
\end{proof}
Now we give the proof of Theorem \ref{TSthm}:\\
\emph{Proof of Theorem \ref{TSthm}}: It is easy to know that for all $i\in \mathcal{V}, t\geq 1$,
\begin{equation}\label{probnoisevalue}
  P\{\xi_i(t)\in[\delta/2,\delta]\}=P\{\xi_i(t)\in[-\delta,-\delta/2]\}=\frac{1}{4}.
\end{equation}
Denote $\widetilde{x}_i(t)=|\mathcal{N}(i, x(t))|^{-1}\sum_{j\in \mathcal{N}(i, x(t))}x_j(t)$, $t\geq 0$, and this denotation remains valid for the rest of the context.
If $d_\mathcal{V}^A(0)\leq \delta$, the conclusion holds by Lemma \ref{tslem} since $\delta_1\geq \delta, \delta_2\geq \delta$. Otherwise, consider the following protocol: for all $i\in\mathcal{V}, t>0$,
\begin{equation}\label{noiseproto}
  \left\{
    \begin{array}{ll}
      \xi_i(t+1)\in[\delta/2,\delta], & \hbox{if}\quad\widetilde{x}_i(t)\leq A; \\
      \xi_i(t+1)\in[-\delta,-\delta/2], & \hbox{if}\quad\widetilde{x}_i(t)>A.
    \end{array}
  \right.
\end{equation}
By Lemma \ref{monosmlem}, (\ref{TSnoise1}) and (\ref{xit}), we know under the protocol (\ref{noiseproto}) that
\begin{equation*}
\begin{split}
  d_\mathcal{V}^A(1)&\leq \max_{i\in\mathcal{V}}\{|\widetilde{x}_i(0)-A|-\delta/2\}\\
&\leq \max_{i\in\mathcal{V}}\{|x_i(0)-A|-\delta/2\}\\
&\leq d_\mathcal{V}^A(0)-\delta/2.
\end{split}
\end{equation*}
By (\ref{probnoisevalue}) and independence,
\begin{equation}\label{probdistt1}
  P\{d_\mathcal{V}^A(1)\leq d_\mathcal{V}^A(0)-\delta/2\}\geq \frac{1}{4^n}.
\end{equation}
Let $L=\frac{1-\delta}{\delta/2}$, if $d_\mathcal{V}^A(1)\leq \delta$, the conclusion holds. Otherwise, continue the above procedure $L$ times and we can get by (\ref{probdistt1}) and independence that
\begin{equation}\label{probdisttL}
    P\{d_\mathcal{V}^A(L)\leq \delta\}\geq \frac{1}{4^{nL}}>0.
\end{equation}
Thus
\begin{equation}\label{probdistnotL}
    P\{d_\mathcal{V}^A(L)> \delta\}\leq 1-\frac{1}{4^{nL}}.
\end{equation}
Denote events
\begin{equation}\label{eventdelt}
\begin{split}
E_0&=\Omega, \\
E_m&=\{\omega: d_\mathcal{V}^A(t)>\delta, (m-1)L<t\leq mL\}, m\geq 1.
\end{split}
\end{equation}
Since $x(0)$ is arbitrarily given, by (\ref{probdistnotL}), we can get for $m\geq 1$ that
\begin{equation}\label{probevent}
\begin{split}
  P\{E_m|E_{m-1}\}&\leq P\{d_\mathcal{V}^A(mL)> \delta|E_{m-1}\}\\
  &\leq 1-\frac{1}{4^{nL}}<1.
  \end{split}
\end{equation}
Note by Lemma \ref{tslem} that $\{\limsup\limits_{t\rightarrow\infty}d_\mathcal{V}^{A}(t)> \bar{\delta}\}\subset\bigg\{\bigcap\limits_{m\geq 1}\{d_\mathcal{V}^A(t)> \delta, (m-1)L<t\leq mL\}\bigg\}$, then by (\ref{probevent})
\begin{equation*}
\begin{split}
&P\{\limsup\limits_{t\rightarrow\infty}d_\mathcal{V}^{A}(t)\leq \bar{\delta}\}\\
=&1-P\{\limsup\limits_{t\rightarrow\infty}d_\mathcal{V}^{A}(t)> \bar{\delta}\}\\
\geq &1-P\bigg\{\bigcap\limits_{m\geq 1}\{d_\mathcal{V}^A(t)> \delta, (m-1)L<t\leq mL\}\bigg\}\\
=& 1-P\bigg\{\bigcap\limits_{m\geq 1}E_m\bigg\} \\
=&1-P\bigg\{\lim\limits_{m\rightarrow \infty}\bigcap\limits_{j= 1}^mE_j\bigg\}=1-\lim\limits_{m\rightarrow \infty}P\bigg\{\bigcap\limits_{j= 1}^mE_j\bigg\} \\
=& 1-\lim\limits_{m\rightarrow \infty}\prod\limits_{j=1}^mP\{E_j|E_{j-1}\}\\
\geq & 1-\lim\limits_{m\rightarrow \infty}\bigg(1-\frac{1}{4^{nL}}\bigg)^m=1.
\end{split}
\end{equation*}
This complete the proof. \hfill $\Box$

\section{Simulations}\label{Secsimul}
In this part, we will present some simulations to demonstrate the truth seeking with the influence of the free information flow. Let $n=20$, $x_i(0), i\in\mathcal{V}$ are randomly generated from the interval $[0,1]$, $\epsilon=0.2, A=0.8, \alpha=0.5, m=10$. Figure \ref{noisefreefig} shows the case of truth seeking when there is no free information flow. It can be seen that only part of the individuals achieve the truth. Next we show the case when there is free information flow. Take the same intrinsic parameters as in Figure \ref{noisefreefig}. According to Corollary \ref{TSScoro}, when the noise strength $0<\delta\leq \epsilon/8$, the system will approach the truth. Take $\delta=0.02$, then Figures \ref{noise1fig}, \ref{noiselogfig} show the result.
\begin{figure}[htp]
  \centering
  \includegraphics[width=2.5in]{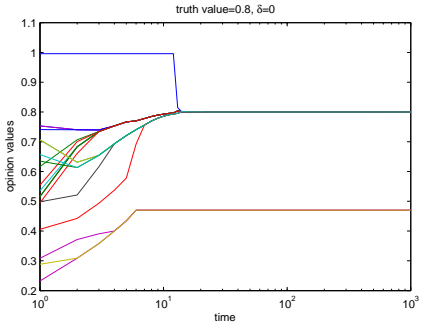}\\
  \caption{The truth seeking of the model (\ref{basicTSmodel}). Take n=20, m=10, $\alpha=0.5$, confidence threshold $\epsilon=0.2$.  }\label{noisefreefig}
\end{figure}
\begin{figure}[htp]
  \centering
  \includegraphics[width=2.5in]{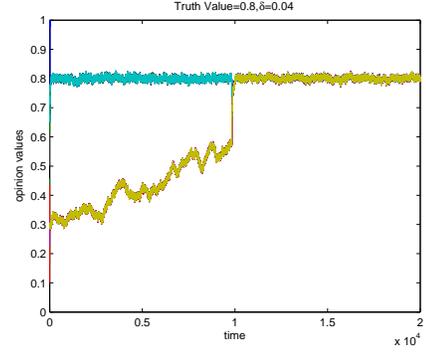}\\
  \caption{ The truth seeking of the model (\ref{TSnoise1})-(\ref{precond}). Take the same parameters and add random noise with strength $\delta=0.02$ here. }\label{noise1fig}
\end{figure}
\begin{figure}[htp]
  \centering
  \includegraphics[width=2.5in]{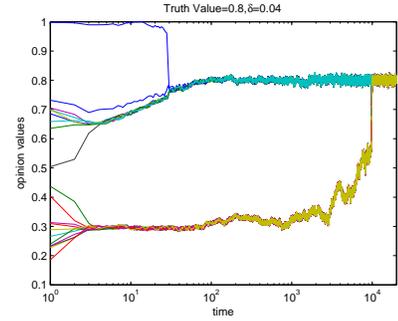}\\
  \caption{ Figure \ref{noise1fig} is re-plotted to better display the initial picture of evolution.  }\label{noiselogfig}
\end{figure}

\section{Conclusions}\label{Secconc}
How the free information flow influence the truth seeking in a society is of interest to the investigation of social dynamics. In this paper, we use a well-established truth seeking model and strictly prove that the free information flow could effectively enable the approach of the truth in a group. This interesting discovery could provide us inspiration to investigate how the free information flow determines the evolution of the social dynamics under other topics.

\end{document}